\documentclass[reqno,11pt]{amsart}
\usepackage{amssymb}
\usepackage{mathrsfs}
\usepackage{}
\usepackage{amsmath,amssymb}
\usepackage{paralist}
\usepackage{graphics} %% add this and next lines if pictures should be in esp format
\usepackage{epsfig} %For pictures: screened artwork should be set up with an 85 or 100 line screen
\usepackage[colorlinks=true]{hyperref}
\hypersetup{urlcolor=red, citecolor=blue}
\usepackage[top=1in,bottom=1in,left=1in,right=1in]{geometry}
\usepackage{cite}
\newtheorem{thm}{Theorem}[section]

\newtheorem{lem}[thm]{Lemma}

\theoremstyle{definition}

\numberwithin{equation}{section}
 \newcommand{\be}{\begin{equation}}
 \newcommand{\ee}{\end{equation}}
 \newcommand\bes{\begin{eqnarray}}
 \newcommand\ees{\end{eqnarray}}
 \newcommand{\bess}{\begin{eqnarray*}}
 \newcommand{\eess}{\end{eqnarray*}}
 
\begin{document}
\title[two species chemotaxis system with signal absorption]
      {Boundedness and stabilization in a two-species chemotaxis system with signal absorption}%
\author[Zhang]{Qingshan Zhang }%
\address{Department of Mathematics, Henan Institute of Science and Technology, Xinxiang 453003, PR China}
\email{qingshan11@yeah.net}

%%%%P13 address changed (we have some new administrative rules)

\author[Tao]{Weirun Tao }%
\address{Institute for Applied Mathematics, School of Mathematics, Southeast University, Nanjing 211189,
PR China}
\email{taoweiruncn@163.com}

%\thanks{Supported in part by National Natural Science Foundation of China (No. 11171063).}


%\date{}%
%\dedicatory{}%
%\commby{}%
% ----------------------------------------------------------------

\begin{abstract}
This paper is concerned with the Neumann initial-boundary value problem for the two-species chemotaxis system with consumption of chemoattractant
\begin{eqnarray*}
  \left\{\begin{array}{lll}
     \medskip
    u_t=\Delta u-\chi_1\nabla\cdot(u\nabla w),&{} x\in\Omega,\ t>0,\\
     \medskip
     v_t=\Delta v-\chi_2\nabla\cdot(v\nabla w),&{} x\in\Omega,\ t>0,\\
     \medskip
     w_t=\Delta w-(\alpha u+\beta v)w,&{}x\in\Omega,\ t>0
  \end{array}\right.
\end{eqnarray*}
in a smooth bounded domain $\Omega\subset\mathbb{R}^n$ ($n\geq2$), where the parameters $\chi_1$, $\chi_2$, $\alpha$ and $\beta$ are positive. It is proved that if
\begin{equation*}
\max\{\chi_1,\chi_2\}\|w(x,0)\|_{L^{\infty}(\Omega)}<\sqrt{\frac{2}{n}}\pi
\end{equation*}
the problem possesses a unique global classical solution that is uniformly bounded. Moreover, we prove that
\begin{equation*}
u(x,t)\to\frac{1}{|\Omega|}\int_{\Omega}u(x,0),\quad v(x,t)\to\frac{1}{|\Omega|}\int_{\Omega}v(x,0)\quad\mbox{and}\quad w(x,t)\to0\quad\mbox{as}\ t\to\infty
\end{equation*}
uniformly with respect $x\in\Omega$.
\end{abstract}
\maketitle
% ----------------------------------------------------------------

\section{Introduction}
In this paper, we deal with the Neumann initial-boundary value problem for the chemotaxis system
\begin{eqnarray}\label{CSC}
  \left\{\begin{array}{lll}
     \medskip
     u_t=\Delta u-\chi_1\nabla\cdot(u\nabla w),&{} x\in\Omega,\ t>0,\\
     \medskip
     v_t=\Delta v-\chi_2\nabla\cdot(v\nabla w),&{} x\in\Omega,\ t>0,\\
     \medskip
     w_t=\Delta w-(\alpha u+\beta v)w,&{}x\in\Omega,\ t>0,\\
     \medskip
     \frac{\partial u}{\partial \nu}=\frac{\partial v}{\partial \nu}=\frac{\partial w}{\partial \nu}=0,&{} x\in\partial\Omega, t>0,\\
     \medskip
     u(x,0)=u_0(x),\quad v(x,0)=v_0(x),\quad w(x,0)=w_0(x),&{}x\in\Omega
  \end{array}\right.
\end{eqnarray}
in the bounded domain $\Omega\subset\mathbb{R}^n$ $(n\geq2)$ with smooth boundary, where $\chi_1$, $\chi_2$, $\alpha$ and $\beta$ are positive constants, the unknown functions $u=u(x,t)$ and $v=v(x,t)$ stand for polulation densities and $w=w(x,t)$ denotes the concentration of chemoattractant. The symbol $\frac{\partial}{\partial \nu}$ represents differentiation with respect to the outward normal $\nu$ on $\partial\Omega$. The initial data $u_0$, $v_0$ and $w_0$ are given positive functions satisfying
\begin{equation}\label{initial data}
u_0\in C^0(\bar{\Omega}),\quad v_0\in C^0(\bar{\Omega})\quad\mbox{and}\quad w_0\in W^{1,q}(\Omega)
\end{equation}
with some $q>n$.

\vskip 3mm

The model (\ref{CSC}) is used in mathematical biology to describe the movement of two populations in respond to the concentration gradient of one common chemical signal. It is a generalization of the famous  Keller-Segel system \cite{Keller&Segel-JTB-1971}
\begin{eqnarray}\label{one species}
  \left\{\begin{array}{lll}
     \medskip
     u_t=\Delta u-\chi\nabla\cdot(u\nabla w),\\
     \medskip
     w_t=\Delta w-uw,
  \end{array}\right.
\end{eqnarray}
%The signal absorption mechanism in (\ref{one species}) has a significant effect to the dynamical behavior of the population.
which has been studied from a mathematical viewpoint in the last years. It is proved that if $n\leq2$ or
\begin{equation*}
\chi\|w(x,0)\|_{L^{\infty}(\Omega)}<\frac{1}{6(n+1)}
\end{equation*}
in higher dimensions $n\geq3$, the problem (\ref{one species}) possesses a unique global classical solution which is bounded and satisfies
\begin{equation}\label{convergence property}
u(x,t)\to\frac{1}{|\Omega|}\int_{\Omega}u_0\quad \mbox{and}\quad w(x,t)\to0\quad \mbox{as}\ t\to\infty
\end{equation}
uniformly with respect to $x\in\Omega$ \cite{Tao-JMAA-2011,Tao&Winkler-JDE-2012,zhang&li-jmp-2015}. Moreover, the problem admits at least one global weak solution which is eventually smooth and enjoys the convergence properties (\ref{convergence property}) in bounded convex domain $\Omega\subset\mathbb{R}^3$ \cite{Tao&Winkler-JDE-2012}. Recently, some blow-up properties of (\ref{CSC}), including blow-up criteria for the local classical solution, lower global blow-up estimate on $\|u\|_{L^{\infty}(\Omega)}$ and local non-degeneracy property for the blow-up points, have been obtained in \cite{Jiang-jde-2018}. Furthermore, global solutions and the stabilization for the corresponding variants of (\ref{one species}), such as chemotaxis-consumption systems with tensor-valued sensitivities \cite{Baghaei-zamm-2018,Xue-M3AS-2015,Winkler-SJMA-2015,zhang-mn-2016} and singular sensitivities \cite{Winkler-M3AS-2016}, system (\ref{one species}) with logistic source \cite{MR3690294,Baghaei-mmas-2017} or coupled chemotaxis-fluid system \cite{Lorz-M3AS-2010,Winkler-ARMA-2014,Winkler-CPDE-2012,winkler2014global} have also been investigated. For more results on the model variations of (\ref{one species}), we refer to the recent survey \cite{Bellomo} and the references therein.

As to the problem (\ref{CSC}) with Lotka-Volterra competitive kinetics
\begin{eqnarray*}
  \left\{\begin{array}{lll}
     \medskip
     u_t=\Delta u-\chi_1\nabla\cdot(u\nabla w)+\mu_1u(1-u-a_1v),&{} x\in\Omega,\ t>0,\\
     \medskip
     v_t=\Delta v-\chi_2\nabla\cdot(v\nabla w)+\mu_2v(1-a_2u-v),&{} x\in\Omega,\ t>0,\\
     \medskip
     w_t=\Delta w-(\alpha u+\beta v)w,&{}x\in\Omega,\ t>0,\\
     \medskip
     \frac{\partial u}{\partial \nu}=\frac{\partial v}{\partial \nu}=\frac{\partial w}{\partial \nu}=0,&{} x\in\partial\Omega, t>0,\\
     \medskip
     u(x,0)=u_0(x),\quad v(x,0)=v_0(x),\quad w(x,0)=w_0(x),&{}x\in\Omega,
  \end{array}\right.
\end{eqnarray*}
it is shown that solutions exist globally and remain bounded if either $n=2$ \cite{Hirata-JDE-2017} or
\begin{equation*}
\chi_i\|w_0\|_{L^{\infty}(\Omega)}<\frac{\pi}{\sqrt{n+1}},\quad i=1,2
\end{equation*}
in the case $n\geq3$ \cite{MR3741393}, and if $a_1$, $a_2\in(0,1)$ these solutions satisfy
\begin{equation*}
(u(\cdot,t), v(\cdot,t), w(\cdot,t))\to\left(\frac{1-a_1}{1-a_1a_2}, \frac{1-a_2}{1-a_1a_2}, 0 \right)\quad\mbox{in}\ L^{\infty}(\Omega)
\end{equation*}
and if $a_1\geq1>a_2$,
\begin{equation*}
(u(\cdot,t), v(\cdot,t), w(\cdot,t))\to(0, 1, 0)\quad\mbox{in}\ L^{\infty}(\Omega)
\end{equation*}
as $t\to\infty$ \cite{Hirata-JDE-2017,MR3741393}. On the other hand, the global existence and stabilization of (weak) solutions to the two-species chemotaxis-fluid system with Lotka-Volterra competitive kinetics have also been   established, see e.g. \cite{Hirata-JDE-2017,Cao-MMAS-2018,Hirata-1710.00957,Jin&Xiang-1706.07910}. However, to the best of our knowledge, the problem (\ref{CSC}) seems not be investigated yet in the literature.

In the present paper, we prove global existence, boundedness and stabilization of classical solutions to the problem (\ref{CSC}). Our main results are stated as follows.

\begin{thm}\label{main result1}
Let $\Omega\subset\mathbb{R}^n$ $(n\geq2)$ be bounded domain with smooth boundary and let the parameters $\chi_1, \chi_2, \alpha, \beta>0$. If
\begin{equation*}
\chi_i\|w_0\|_{L^{\infty}(\Omega)}<\sqrt{\frac{2}{n}}\pi,\quad i=1,2,
\end{equation*}
then for any initial data $u_0$, $v_0$ and $w_0$ satisfying {\rm{(\ref{initial data})}}, the problem {\rm{(\ref{CSC})}} possesses a unique global classical solution which is bounded in $\Omega\times(0, \infty)$.
\end{thm}
The second result concerns stabilization of the solution provided by Theorem \ref{main result1}.
\begin{thm}\label{main result2}
Under the assumptions of Theorem \ref{main result1}, the global classical solution of {\rm{(\ref{CSC})}} satisfies
\begin{eqnarray*}
&&\|u(\cdot,t)-\bar{u}_0\|_{L^{\infty}(\Omega)}\to0,\\
&&\|v(\cdot,t)-\bar{v}_0\|_{L^{\infty}(\Omega)}\to0,\\
&&\|w(\cdot,t)\|_{L^{\infty}(\Omega)}\to0
\end{eqnarray*}
as $t\to\infty$, where $\bar{u}_0:=\frac{1}{|\Omega|}\int_{\Omega}u_0$ and $\bar{v}_0:=\frac{1}{|\Omega|}\int_{\Omega}v_0$.
\end{thm}
The rest of the paper is organized as follows. In Section 2, we list some preliminaries. Section 3 is devoted to the proof of Theorem \ref{main result1}. Finally, we prove Theorem \ref{main result2} in Section 4.

\section{Preliminaries}
As a preparation to the proof, we first state one result concerning local existence of classical solution to the problem (\ref{CSC}).
\begin{lem}\label{local existence}
Let $u_0$, $v_0$ and $w_0$ satisfy {\rm{(\ref{initial data})}}, and let $\chi_1, \chi_2, \alpha, \beta>0$. Then there exist $T_{\max}\leq\infty$ and a uniquely determined triple $(u,v, w)$ of nonnegative functions
\begin{eqnarray*}
&&u\in C^0(\bar{\Omega}\times[0, T_{\max}))\cap C^{2,1}(\bar{\Omega}\times(0, T_{\max})),\\
&&v\in C^0(\bar{\Omega}\times[0, T_{\max}))\cap C^{2,1}(\bar{\Omega}\times(0, T_{\max})),\\
&&w\in C^0(\bar{\Omega}\times[0, T_{\max}))\cap C^{2,1}(\bar{\Omega}\times(0, T_{\max}))\cap L_{loc}^{\infty}([0, T_{\max});W^{1,q}(\Omega)),
\end{eqnarray*}
which solves {\rm{(\ref{CSC})}} in the classical sense. If $T_{\max}<\infty$, then
 \begin{equation*}
\|u(t)\|_{L^{\infty}(\Omega)}+\|v(t)\|_{L^{\infty}(\Omega)}+\|w(t)\|_{W^{1,q}(\Omega)}\to\infty\quad\mbox{as}\ t\nearrow T_{\max}.
\end{equation*}
Furthermore, the solution $(u,v,w)$ satisfies
\begin{equation}\label{mc}
\|u(t)\|_{L^1{(\Omega)}}=\|u_0\|_{L^1{(\Omega)}}\quad \mbox{for all}\ t\in(0, T_{max})
\end{equation}
and
\begin{equation}\label{mcv}
\|v(t)\|_{L^1{(\Omega)}}=\|v_0\|_{L^1{(\Omega)}}\quad \mbox{for all}\ t\in(0, T_{max})
\end{equation}
as well as
\begin{equation}\label{bound w}
0\leq w\leq\|w_0\|_{L^{\infty}(\Omega)}\quad \mbox{in}\ \bar\Omega\times[0, T_{max}).
\end{equation}
\end{lem}
\begin{proof}
The local existence and regularity of the solution is based on standard contraction mapping arguments and parabolic regularity theory, which can be found in \cite [Lemma 1.1]{Winkler-CPDE-2010}(see also \cite [Lemma 2.1]{Tao-JMAA-2011}). Integrating the first and the second equations in (\ref{CSC}), we immediately obtain
\begin{equation*}
\frac{d}{dt}\int_{\Omega}u(x,t)dx=0\quad\mbox{and}\quad\frac{d}{dt}\int_{\Omega}v(x,t)dx=0
\end{equation*}
which yield (\ref{mc}) and (\ref{mcv}). The statement (\ref{bound w}) follows from an application of the maximum principle.
\end{proof}

The following lemma is a generalization of \cite[Lemma 3.2]{zhang-mn-2016}, which asserts that a bound for $\|u\|_{L^p(\Omega)}$ and $\|v\|_{L^p(\Omega)}$ with $p>\frac{n}{2}$ for all $t\in(0, T_{\max})$ can guarantee the global existence and boundedness of classical solutions to (\ref{CSC}).
\begin{lem}\label{improved extensibility}
Suppose that the initial data $u_0$, $v_0$ and $w_0$ satisfy {\rm{(\ref{initial data})}} and the parameters $\chi_1, \chi_2, \alpha, \beta>0$. Let $p\geq1$. If the first and second components of solution satisfy
\begin{equation}\label{up assumption}
\sup_{t\in(0, T_{\max})}\left(\|u(t)\|_{L^p(\Omega)}+\|v(t)\|_{L^p(\Omega)}\right)<\infty,
\end{equation}
for some $p>\frac{n}{2}$, then the solution of {\rm{(\ref{CSC})}} is global in time. Moreover, the solution fulfills
\begin{equation*}
\sup_{t\in(0, \infty)}\left(\|u(t)\|_{L^{\infty}(\Omega)}+\|v(t)\|_{L^{\infty}(\Omega)}+\|w(t)\|_{W^{1,q}(\Omega)}\right)<\infty.
\end{equation*}
\end{lem}

\begin{proof}
Since our assumption $q>n$ and for each fixed $p>\frac{n}{2}$ there holds
\begin{eqnarray*}
  \frac{np}{(n-p)_{+}}=\left\{\begin{array}{lll}
     \medskip
     \infty,&{} \mbox{if}\ p\geq n,\\
     \medskip
      \frac{np}{n-p}>n,&{} \mbox{if}\ \frac{n}{2}<p<n,
  \end{array}\right.
\end{eqnarray*}
it is possible to find $1<p_0<q$ fulfilling
\begin{equation}\label{p1}
n<p_0<\frac{np}{(n-p)_{+}},
\end{equation}
which enables us to choose $k>1$ such that
\begin{equation}\label{p2}
n<kp_0<\frac{np}{(n-p)_{+}}\quad\mbox{and}\quad kp_0\leq q.
\end{equation}
We shall argue by contradiction. Assume that $T_{\max}<\infty$.  Applying the variation-of-constants formula
\begin{equation*}
w(\cdot,t)=e^{t\Delta}w_0-\int_0^te^{(t-s)\Delta}(\alpha u(\cdot,s)+\beta v(\cdot,s))w(\cdot,s)ds,
\end{equation*}
we get
\begin{equation*}
\|\nabla w(\cdot,t)\|_{L^{kp_0}(\Omega)}\leq\left\|\nabla e^{t\Delta}w_0\right\|_{L^{kp_0}(\Omega)}+\int_0^t
    \left\|\nabla e^{(t-s)\Delta}(\alpha u(\cdot,s)+\beta v(\cdot,s))w(\cdot,s)\right\|_{L^{kp_0}(\Omega)}ds
\end{equation*}
for all $t\in(0, T_{\max})$. By standard smoothing estimates for the Neumann heat semigroup \cite[Lemma 1.3]{Winkler-JDE-2010}, (\ref{bound w}) and $kp_0\leq q$ due to (\ref{p2}), we obtain positive constants $C_1$ and $C_2$ such that
\begin{eqnarray*}
\|\nabla w(\cdot,t)\|_{L^{kp_0}(\Omega)}&\leq&C_1\left\|w_0\right\|_{W^{1,q}(\Omega)}+C_2\int_0^t\left(1+(t-s)^{-\frac{1}{2}-\frac{n}{2}(\frac{1}{p}-\frac{1}{k p_0})}\right)e^{-\lambda_1(t-s)}\\
                                              &\quad&\times\left\|(\alpha u(\cdot,s)+\beta v(\cdot,s))w(\cdot,s)\right\|_{L^p(\Omega)}ds\\
                                        &\leq&C_1\left\|w_0\right\|_{W^{1,q}(\Omega)}+C_2\|w_0\|_{L^{\infty}(\Omega)}\int_0^t\left(1+(t-s)^{-\frac{1}{2}-\frac{n}{2}(\frac{1}{p}-\frac{1}{k p_0})}\right)e^{-\lambda_1(t-s)}\\
                                              &\quad&\times\left(\|u(\cdot,s)\|_{L^p(\Omega)}+\|v(\cdot,s)\|_{L^p(\Omega)}\right)ds
\end{eqnarray*}
for all $t\in(0, T_{\max})$. Here and below, $\lambda_1>0$ denotes the first nonzero eigenvalue of $-\Delta$ in $\Omega$ under homogeneous Neumann boundary conditions. Since $\frac{1}{2}+\frac{n}{2}(\frac{1}{p}-\frac{1}{k p_0})<1$ by the right-hand side of the first inequality in (\ref{p2}) and (\ref{up assumption}), we can take constants $C_3>0$ and $C_4>0$ fulfilling
\begin{eqnarray}\label{nabla w}
\|\nabla w(\cdot,t)\|_{L^{kp_0}(\Omega)}&\leq&C_1\left\|w_0\right\|_{W^{1,q}(\Omega)}+C_3\left(\sup_{t\in(0, T_{\max})}\left(\|u(t)\|_{L^p(\Omega)}+\|v(t)\|_{L^p(\Omega)}\right)\right)\nonumber\\
                                        &\leq&C_4\quad \mbox{for all}\ t\in(0, T_{max}).
\end{eqnarray}
Next by the variation-of-constants formula
\begin{equation*}
u(\cdot,t)=e^{t\Delta}u_0-\chi_1\int_0^t\nabla e^{(t-s)\Delta}u(\cdot,s)\nabla w(\cdot,s)ds,
\end{equation*}
we have
\begin{eqnarray*}
\|u(\cdot,t)\|_{L^{\infty}(\Omega)}&\leq&\left\|e^{t\Delta}u_0\right\|_{L^{\infty}(\Omega)}+\chi_1\int_0^t\big\|\nabla e^{(t-s)\Delta}u(\cdot,s)\nabla w(\cdot,s)\big\|_{L^{\infty}(\Omega)}ds
\end{eqnarray*}
for all $t\in(0, T_{\max})$. In view of the maximum principle and smoothing estimates for the Neumann heat semigroup \cite[Lemma 1.3]{Winkler-JDE-2010}, we obtain $C_5>0$ satisfying
\begin{eqnarray}\label{u finity1}
\|u(\cdot,t)\|_{L^{\infty}(\Omega)}\leq\left\|u_0\right\|_{L^{\infty}(\Omega)}+C_5\int_0^t\left(1+(t-s)^{-\frac{1}{2}-\frac{n}{2 p_0}}\right)e^{-\lambda_1(t-s)}\big\|u(\cdot,s)\nabla w(\cdot,s)\big\|_{L^{p_0}(\Omega)}ds
\end{eqnarray}
for all $t\in(0, T_{\max})$. Here by the H\"{o}lder inequality, interpolation inequality, (\ref{mc}) and (\ref{nabla w}), we can find find $C_6>0$ such that
\begin{eqnarray*}
\big\|u(\cdot,s)\nabla w(\cdot,s)\big\|_{L^{p_0}(\Omega)}&\leq&\big\|u(\cdot,s)\big\|_{L^{k'p_{0}}(\Omega)}\big\|\nabla w(\cdot,s)\big\|_{L^{k p_0}(\Omega)}\nonumber\\
&\leq&\big\|u(\cdot,s)\big\|^{r}_{L^{\infty}(\Omega)}\big\|u(\cdot,s)\big\|^{1-r}_{L^{1}(\Omega)}\big\|\nabla w(\cdot,s)\big\|_{L^{k p_0}(\Omega)}\nonumber\\
&\leq&C_6\big\|u(\cdot,s)\big\|^{r}_{L^{\infty}(\Omega)}\quad\mbox{for all}\ s\in(0, T_{\max}),
\end{eqnarray*}
where $k'$ is the dual exponent of $k$ and $r=1-\frac{1}{k'p_0}\in(0,1)$. Inserting this into (\ref{u finity1}), it follows that
\begin{eqnarray*}
\sup_{t\in(0,T)}\|u(\cdot,t)\|_{L^{\infty}(\Omega)}\leq\left\|u_0\right\|_{L^{\infty}(\Omega)}+
C_7\sup_{t\in(0,T)}\|u(\cdot,t)\|^{r}_{L^{\infty}(\Omega)}\quad\mbox{for all}\ T\in(0, T_{\max})
\end{eqnarray*}
with
\begin{equation*}
C_7=C_5C_6\int_0^{\infty}\left(1+\sigma^{-\frac{1}{2}-\frac{n}{2 p_0}}\right)e^{-\lambda_1\sigma}d\sigma
\end{equation*}
is finite thanks to the left-hand side of (\ref{p1}). Therefore, there exists a constant $C_8>0$ such that
\begin{eqnarray*}
\|u(\cdot,t)\|_{L^{\infty}(\Omega)}\leq C_8\quad\mbox{for all}\ t\in(0, T_{\max}).
\end{eqnarray*}
Arguing similarly as above, we see that
\begin{eqnarray*}
\|v(\cdot,t)\|_{L^{\infty}(\Omega)}\leq C_{9}\quad\mbox{for all}\ t\in(0, T_{\max}).
\end{eqnarray*}
with some $C_{9}>0$. This is a contradiction to Lemma \ref{local existence}. Hence we complete the proof.
\end{proof}

For the proof of the main result we also need the following technical lemma which provides the desired weight function.
\begin{lem}\label{weight function}
Let $\varepsilon\in(0,1)$ and $p>1$. Define the function
\begin{equation*}
\varphi(s):=e^{z(s)},\quad0\leq s\leq M,
\end{equation*}
where
\begin{equation*}
z(s):=-\frac{b}{2c}s+\frac{\sqrt{4ac-b^2}}{2c}\int_0^s\tan\left(\frac{\sqrt{4ac-b^2}}{2d}\tau+\arctan\frac{b}{\sqrt{4ac-b^2}}\right)d\tau
\end{equation*}
with
$a=(p-1)^2,\ b=-4(p-1)\varepsilon,\ c=\frac{4}{p}(1+(p-1)\varepsilon)$ and $d=\frac{4}{p}(p-1)(1-\varepsilon)$. If there holds
\begin{equation}\label{condition}
M<\frac{2}{\sqrt{p}}\sqrt{\frac{1-\varepsilon}{1+p\varepsilon}}\left(\frac{\pi}{2}+\arctan\sqrt{\frac{p}{1+(p-1)\varepsilon-p\varepsilon^2}}\varepsilon\right),
\end{equation}
then the function $\varphi(s)$ is well defined and satisfies the following conditions
\begin{equation}\label{con1}
\varphi'(s)\geq0,
\end{equation}
\begin{equation}\label{con2}
1\leq\varphi(s)\leq\varphi(M),
\end{equation}
\begin{equation}\label{con3}
\frac{1}{p}\varphi''(s)-\varphi'(s)\geq0
\end{equation}
and
\begin{equation}\label{con4}
|(p-1)\varphi(s)-2\varphi'(s)|-2\sqrt{(p-1)(1-\varepsilon)\varphi(s)\left(\frac{1}{p}\varphi''(s)-\varphi'(s)\right)}=0
\end{equation}
for all $0\leq s\leq M$.
\end{lem}
\begin{proof}
First of all we remark that for any $\varepsilon\in(0,1)$
\begin{equation*}
4ac-b^2=\frac{16(p-1)^2}{p}\left(1+(p-1)\varepsilon-p\varepsilon^2\right)>0.
\end{equation*}
Owing to (\ref{condition}), for $0\leq s\leq M$ we get
\begin{eqnarray*}
&&\frac{\sqrt{4ac-b^2}}{2d}s+\arctan\frac{b}{\sqrt{4ac-b^2}}\geq-\arctan\sqrt{\frac{p}{1+(p-1)\varepsilon-p\varepsilon^2}}\varepsilon>-\frac{\pi}{2}
\end{eqnarray*}
and
\begin{eqnarray*}
&&\frac{\sqrt{4ac-b^2}}{2d}s+\arctan\frac{b}{\sqrt{4ac-b^2}}\\
&&\leq\frac{\sqrt{p}}{2}\sqrt{\frac{1+p\varepsilon}{1-\varepsilon}}M-\arctan\sqrt{\frac{p}{1+(p-1)\varepsilon-p\varepsilon^2}}\varepsilon\\
&&<\frac{\sqrt{p}}{2}\sqrt{\frac{1+p\varepsilon}{1-\varepsilon}}\left(\frac{2}{\sqrt{p}}\sqrt{\frac{1-\varepsilon}{1+p\varepsilon}}\left(\frac{\pi}{2}
   +\arctan\sqrt{\frac{p}{1+(p-1)\varepsilon-p\varepsilon^2}}\varepsilon\right)\right)\\
&&\quad-\arctan\sqrt{\frac{p}{1+(p-1)\varepsilon-p\varepsilon^2}}\varepsilon\\
&&=\frac{\pi}{2}.
\end{eqnarray*}
Therefore, the function $\varphi(s)$ is well defined for $0\leq s\leq M$. A direct computation reveals that
\begin{eqnarray*}
\varphi'(s)&=&\varphi(s)z'(s)\\
           &=&\varphi(s)\left(-\frac{b}{2c}+\frac{\sqrt{4ac-b^2}}{2c}\left(\tan\left(\frac{\sqrt{4ac-b^2}}{2d}s+\arctan\frac{b}{\sqrt{4ac-b^2}}\right)\right)\right)\\
           &\geq&\varphi(s)z'(0)\\
           &=&0\quad\mbox{for all}\ 0\leq s\leq M.
\end{eqnarray*}
Then the relation (\ref{con2}) is obvious according to (\ref{con1}). Since $\varphi''(s)=\varphi(s)\left(z''(s)+(z'(s))^2\right)$ and
\begin{eqnarray*}
z''(s)&=&\frac{4ac-b^2}{4cd}\left(1+\tan^2\left(\frac{\sqrt{4ac-b^2}}{2d}s+\arctan\frac{b}{\sqrt{4ac-b^2}}\right)\right)\\
      &=&\frac{4ac-b^2}{4cd}\left(1+\frac{4c^2}{4ac-b^2}\left(z'(s)+\frac{b}{2c}\right)^2\right)\\
      &=&\frac{1}{d}\left(a+bz'(s)+c(z'(s))^2\right),
\end{eqnarray*}
we have
\begin{eqnarray*}
\frac{1}{p}\varphi''(s)-\varphi'(s)&=&\left(\frac{1}{p}\left((z''(s)+(z'(s))^2\right)-z'(s)\right)\varphi(s)\\
                                   &=&\frac{1}{p}\left(\left(\frac{c}{d}+1\right)(z'(s))^2+\left(\frac{b}{d}-p\right)z'(s)+\frac{a}{d}\right)\varphi(s)\\
                                   &=&\frac{1}{4(p-1)(1-\varepsilon)}\varphi(s)\big((p-1)-2z'(s)\big)^2\\
                                   &=&\frac{1}{4(p-1)(1-\varepsilon)}\frac{1}{\varphi(s)}\big((p-1)\varphi(s)-2\varphi'(s)\big)^2,
\end{eqnarray*}
which proves (\ref{con3}) and (\ref{con4}).
\end{proof}

\section{Global existence. Proof of Theorem \ref{main result1}}
With the preliminaries at hand, we are now prepared to prove global existence and boundedness in (\ref{CSC}). In the following lemma we establish $L^p$-bounds for $u$. The method of the proof is a modification of an idea in \cite{Winkler-MN-2010} (see also, e.g. \cite{Baghaei-mmas-2017,Tao-JMAA-2011,zhang-mn-2016}).
\begin{lem}\label{u lp}
Let $p>1$ and $\varepsilon\in(0,1)$. Assume that the initial functions $u_0$, $v_0$ and $w_0$ fulfill {\rm{(\ref{initial data})}}. If the assumption {\rm{(\ref{condition})}} holds with $M=\max\{\chi_1, \chi_2\}\|w_0\|_{L^{\infty}(\Omega)}$, then there exists $C(p, \varepsilon)>0$ such that for each $\varepsilon\in(0,1)$ the first and second components of the solution for the system {\rm{(\ref{CSC})}} satisfy
\begin{equation}\label{up}
\|u(\cdot, t)\|_{L^{p}(\Omega)}\leq C(p, \varepsilon)\quad\mbox{for all}\ t\in(0, T_{\max})
\end{equation}
and
\begin{equation}\label{vp}
\|v(\cdot, t)\|_{L^{p}(\Omega)}\leq C(p, \varepsilon)\quad\mbox{for all}\ t\in(0, T_{\max}).
\end{equation}
\end{lem}

\begin{proof}
Let $\varphi(s)$ be the function defined in Lemma \ref{weight function} and set $w_1:=\chi_1w$. Using the first equation and the third equation in (\ref{CSC}), we integrate by parts to obtain
 \begin{eqnarray*}
   \frac{1}{p}\frac{d}{dt}\int_{\Omega}u^p\varphi(w_1)&=&\int_{\Omega}u^{p-1}\varphi(w_1)u_t+\frac{1}{p}\int_{\Omega}u^p\varphi'(w_1)(w_1)_{t}\\
    &=&\int_{\Omega}u^{p-1}\varphi(w_1)(\Delta u-\nabla\cdot(u\nabla w_1))+\frac{1}{p}\int_{\Omega}u^p\varphi'(w_1)(\Delta w_1-(\alpha u+\beta v)w_1)\\
    &=&-(p-1)\int_{\Omega}u^{p-2}\varphi(w_1)|\nabla u|^2-\int_{\Omega}u^{p-1}\varphi'(w_1)\nabla u\cdot\nabla w_1\\
     &\quad&+(p-1)\int_{\Omega}u^{p-1}\varphi(w_1)\nabla u\cdot\nabla w_1+\int_{\Omega}u^{p}\varphi'(w_1)|\nabla w_1|^2\\
     &\quad&-\int_{\Omega}u^{p-1}\varphi'(w_1)\nabla u\cdot\nabla w_1-\frac{1}{p}\int_{\Omega}u^p\varphi''(w_1)|\nabla w_1|^2\\
     &\quad&-\frac{1}{p}\int_{\Omega}u^{p}\varphi'(w_1)(\alpha u+\beta v)w_1
 \end{eqnarray*}
for all $t\in(0, T_{\max})$. Fix $\varepsilon\in(0,1)$. Due to (\ref{con1}), it follows that
 \begin{eqnarray}\label{lp1}
   &&\frac{1}{p}\frac{d}{dt}\int_{\Omega}u^p\varphi(w_1)+(p-1)\varepsilon\int_{\Omega}u^{p-2}\varphi(w_1)|\nabla{u}|^2\nonumber\\
   &&\leq-(p-1)(1-\varepsilon)\int_{\Omega}u^{p-2}\varphi(w_1)|\nabla{u}|^2+\int_{\Omega}|(p-1)\varphi(w_1)-2\varphi'(w_1)|u^{p-1}|\nabla u||\nabla w_1|\nonumber\\
   &&\quad-\int_{\Omega}\left(\frac{1}{p}\varphi''(w_1)-\varphi'(w_1)\right)u^p|\nabla w_1|^2\quad\mbox{for all}\ t\in(0, T_{\max,\varepsilon}).
 \end{eqnarray}
Applying (\ref{mc}), (\ref{con2}), the Gagliardo-Nirenberg inequality and Young's inequality, we can find $C_1>0$ and $C_2>0$ such that
\begin{eqnarray*}
\int_{\Omega}u^p\varphi(w_1)&\leq&\varphi\left(\chi_1\|w_0\|_{L^{\infty}(\Omega)}\right)\left\|u^{\frac{p}{2}}\right\|^2_{L^2(\Omega)}\\
&\leq&C_1\Big(\left\|\nabla u^{\frac{p}{2}}\right\|^{2a}_{L^2(\Omega)}\left\|u^{\frac{p}{2}}\right\|^{2(1-a)}_{L^{\frac{2}{p}}(\Omega)}+\left\|u^{\frac{p}{2}}\right\|^{2}_{L^{\frac{2}{p}}(\Omega)}\Big)\\
&\leq&\frac{4(p-1)}{p^2}\varepsilon\left\|\nabla u^{\frac{p}{2}}\right\|^2_{L^2(\Omega)}+C_2\\
&\leq&(p-1)\varepsilon\int_{\Omega}u^{p-2}\varphi(w_1)|\nabla u|^2+C_2
\end{eqnarray*}
with
\begin{equation*}
a=\frac{\frac{p}{2}-\frac{1}{2}}{\frac{p}{2}+\frac{1}{n}-\frac{1}{2}}\in(0,1).
\end{equation*}
This combined with (\ref{lp1}) implies
\begin{eqnarray*}
   &&\frac{1}{p}\frac{d}{dt}\int_{\Omega}u^p\varphi(w_1)+\int_{\Omega}u^p\varphi(w_1)\\
   &&\leq-(p-1)(1-\varepsilon)\int_{\Omega}u^{p-2}\varphi(w_1)|\nabla{u}|^2+\int_{\Omega}|(p-1)\varphi(w_1)-2\varphi'(w_1)|u^{p-1}|\nabla u||\nabla w_1|\\
   &&\quad-\int_{\Omega}\left(\frac{1}{p}\varphi''(w_1)-\varphi'(w_1)\right)u^p|\nabla w_1|^2+C_2\quad\mbox{for all}\ t\in(0, T_{\max}).
\end{eqnarray*}
By (\ref{con3}), we can rewrite it as
\begin{eqnarray*}
   &&\frac{1}{p}\frac{d}{dt}\int_{\Omega}u^p\varphi(w_1)+\int_{\Omega}u^p\varphi(w_1)\\
   &&\leq-\int_{\Omega}\left(\sqrt{(p-1)(1-\varepsilon)\varphi(w_1)}u^{\frac{p-2}{2}}|\nabla{u}|-\sqrt{\left(\frac{1}{p}\varphi''(w_1)-\varphi'(w_1)\right)}u^{\frac{p}{2}}|\nabla w_1|\right)^2\\
   &&\quad+\int_{\Omega}\Phi(w_1)u^{p-1}|\nabla u||\nabla w_1|+C_2\\
   &&\leq C_2\quad\mbox{for all}\ t\in(0, T_{\max}),
\end{eqnarray*}
where
\begin{equation*}
\Phi(w_1):=|(p-1)\varphi(w_1)-2\varphi'(w_1)|-2\sqrt{(p-1)(1-\varepsilon)\varphi(w_1)\left(\frac{1}{p}\varphi''(w_1)-\varphi'(w_1)\right)}=0
\end{equation*}
for all $0\leq w_1\leq\chi_1\|w_0\|_{L^{\infty}(\Omega)}$ due to (\ref{con4}). Therefore, in view of a standard ODE argument, we prove (\ref{up}). Using the second equation and the third equation in (\ref{CSC}), proceeding similarly as for (\ref{up}), we get (\ref{vp}) and hence completes the proof.
\end{proof}

Our main result on the boundedness of the solution in this paper is an immediate consequence of Lemma \ref{improved extensibility} and Lemma \ref{u lp}.

\vskip 3mm

\noindent{\it{Proof of Theorem \ref{main result1}.}} In view of Lemma \ref{improved extensibility} and Lemma \ref{u lp}, it is suffices to verify (\ref{condition}) with $M=M_i:=\chi_i\|w_0\|_{L^{\infty}(\Omega)}$, some $p_i>\frac{n}{2}$ and fixed $\varepsilon_i\in(0,1)$ for $i=1,2$. Since by our assumption $M_i<\sqrt{\frac{2}{n}}\pi$,
we can choose
\begin{equation*}
\varepsilon_i:=\frac{\pi^2-\frac{n}{2}M_i^2}{2\left(\pi^2+\left(\frac{n}{2}\right)^2M_i^2\right)}\in(0,1)
\end{equation*}
such that
\begin{equation*}
M_i=\frac{2}{\sqrt{\frac{n}{2}}}\sqrt{\frac{1-2\varepsilon_i}{1+2\varepsilon_i\frac{n}{2}}}\frac{\pi}{2}
<\frac{2}{\sqrt{\frac{n}{2}}}\sqrt{\frac{1-\varepsilon_i}{1+\varepsilon_i\frac{n}{2}}}\frac{\pi}{2}
\end{equation*}
for $i=1,2$. So that there exists $p_i>\frac{n}{2}$ such that
\begin{eqnarray*}
M_i&=&\frac{2}{\sqrt{p_i}}\sqrt{\frac{1-\varepsilon_i}{1+\varepsilon_ip_i}}\frac{\pi}{2}\\
                            &<&\frac{2}{\sqrt{p_i}}\sqrt{\frac{1-\varepsilon_i}{1+p\varepsilon_i}}\left(\frac{\pi}{2}+\arctan\sqrt{\frac{p_i}{1+(p_i-1)\varepsilon_i-p_i\varepsilon_i^2}}\varepsilon_i\right)
\end{eqnarray*}
for $i=1,2$. This completes the proof.$\hfill\Box$

\section{Stabilization. Proof of Theorem \ref{main result2}}
In this section we consider the asymptotic behavior of the global bounded solutions to (\ref{CSC}).
\begin{lem}
Let the assumptions in Theorem \ref{main result1} hold. Then there exists a constant $C>0$ such that
\begin{equation}\label{estimate1}
\int_0^{\infty}\int_{\Omega}|\nabla u|^2\leq C\quad\mbox{and}\quad\int_0^{\infty}\int_{\Omega}|\nabla v|^2\leq C.
\end{equation}
\end{lem}

\begin{proof}
Multiplying the third equation in (\ref{CSC}) by $w$ and integrating by parts, we obtain
\begin{equation*}
\frac{1}{2}\frac{d}{dt}\int_{\Omega}w^2+\int_{\Omega}|\nabla w|^2=-\int_{\Omega}(\alpha u+\beta v)w^2\quad\mbox{for all}\ t>0.
\end{equation*}
Since $\alpha$, $\beta$, $u$, $v$ and $w$ are all nonnegative, a time integration over $(0,T)$ yields
\begin{equation}\label{estimate nabla w}
\int_0^{T}\int_{\Omega}|\nabla w|^2\leq\frac{1}{2}\int_{\Omega}w_0^2\quad\mbox{for all}\ T>0.
\end{equation}
We respectively test the first and second equations in (\ref{CSC}) by $u$ and $v$ to obtain
\begin{eqnarray*}
\frac{1}{2}\frac{d}{dt}\int_{\Omega}u^2+\int_{\Omega}|\nabla u|^2&=&\chi_1\int_{\Omega}u\nabla u\cdot \nabla w\\
                                                                 &\leq&\frac{1}{2}\int_{\Omega}|\nabla u|^2
                                                                        +\frac{1}{2}\chi^2_1\|u\|^2_{L^{\infty}(\Omega\times(0,\infty))}\int_{\Omega}|\nabla w|^2
\end{eqnarray*}
and
\begin{eqnarray*}
\frac{1}{2}\frac{d}{dt}\int_{\Omega}v^2+\int_{\Omega}|\nabla v|^2&=&\chi_2\int_{\Omega}v\nabla v\cdot \nabla w\\
                                                                 &\leq&\frac{1}{2}\int_{\Omega}|\nabla v|^2
                                                                        +\frac{1}{2}\chi^2_2\|v\|^2_{L^{\infty}(\Omega\times(0,\infty))}\int_{\Omega}|\nabla w|^2
\end{eqnarray*}
for all $t>0$, where we have used Young's inequality. Integrating over $(0,T)$ and applying (\ref{estimate nabla w}) readily imply (\ref{estimate1}).
\end{proof}
Based on the parabolic regularity argument \cite[Lemma 4.3]{Xue-M3AS-2015}, we have the following H\"{o}lder estimates of $u$ and $v$.
\begin{lem}\label{lem42}
Under the assumptions of Theorem \ref{main result1}, there exists $\theta>0$ and $C>0$ such that
\begin{equation}\label{estimate1}
\|u\|_{C^{\theta,\frac{\theta}{2}}(\bar{\Omega}\times[t,t+1])}\leq C\quad\mbox{and}\quad\|v\|_{C^{\theta,\frac{\theta}{2}}(\bar{\Omega}\times[t,t+1])}\leq C
\end{equation}
for all $t\geq1$.
\end{lem}
\begin{proof}
We first note that
\begin{eqnarray*}
(\nabla u-\chi_1u\nabla w)\cdot\nabla u&=&|\nabla u|^2-\chi_1u\nabla u\cdot\nabla w\\
                                       &\geq&\frac{1}{2}|\nabla u|^2-\frac{1}{2}\chi_1^2\|u\|^2_{L^{\infty}(\Omega\times(0,\infty))}|\nabla w|^2
\end{eqnarray*}
and
\begin{equation*}
|\nabla u-\chi_1u\nabla w|\leq|\nabla u|+\chi_1\|u\|_{L^{\infty}(\Omega\times(0,\infty))}|\nabla w|.
\end{equation*}
Since $|\nabla w|^2\in L^{\infty}((0,\infty);L^{\frac{q}{2}})$ with $q>n$, an application of standard parabolic regularity estimates \cite[Theorem 1.3]{Porzio-JDE-1993} yields $\theta_1>0$ and $C_1>0$ such that
\begin{equation*}
\|u\|_{C^{\theta_1,\frac{\theta_1}{2}}(\bar{\Omega}\times[t,t+1])}\leq C_1\quad\mbox{for all}\ t\geq1.
\end{equation*}
By a similar reasoning, we have $\theta_2>0$ and $C_2>0$ fulfilling
\begin{equation*}
\|v\|_{C^{\theta_2,\frac{\theta_2}{2}}(\bar{\Omega}\times[t,t+1])}\leq C_2\quad\mbox{for all}\ t\geq1.
\end{equation*}
This completes the proof.
\end{proof}

In the proof of the stabilization of the solution, we shall need the following common statement which is proved in \cite[Lemma 4.6]{Hirata-JDE-2017}.
\begin{lem}\label{lem43}
Let $n\in C^0(\bar{\Omega}\times[0,\infty))$ satisfy that there exist $C>0$ and $\theta\in(0,1)$ such that
\begin{equation*}
\|n\|_{C^{\theta,\frac{\theta}{2}}(\bar{\Omega}\times[t,t+1])}\leq C\quad\mbox{for all}\ t\geq1.
\end{equation*}
If
\begin{equation*}
\int_0^{\infty}\int_{\Omega}(n(x,t)-N)^2<\infty
\end{equation*}
with some constant $N>0$, then we have
\begin{equation*}\label{estimate1}
n(\cdot,t)\to N\quad \mbox{in}\ C^0(\bar{\Omega})\quad\mbox{as}\quad\ t\to\infty.
\end{equation*}
\end{lem}

\vskip 3mm
With the above lemmas at hand, we are now in the position to prove our main result on stabilization of solutions.

\noindent{\it{Proof of Theorem \ref{main result2}.}} We apply the Poincar\'e inequality
\begin{equation*}
\left\|\varphi-\frac{1}{|\Omega|}\int_{\Omega}\varphi\right\|^2_{L^2(\Omega)}\leq C(n,\Omega)\|\nabla\varphi\|^2_{L^2(\Omega)},\quad\varphi\in W^{1,2}(\Omega)
\end{equation*}
and (\ref{estimate1}) to get $C_1>0$ such that
\begin{eqnarray*}
\int_0^{\infty}\int_{\Omega}|u(x,t)-\bar{u}_0|^2&\leq&C(n,\Omega)\int_0^{\infty}\int_{\Omega}|\nabla u|^2\\
                                                &\leq&C_1
\end{eqnarray*}
and
\begin{eqnarray*}
\int_0^{\infty}\int_{\Omega}|v(x,t)-\bar{v}_0|^2&\leq&C(n,\Omega)\int_0^{\infty}\int_{\Omega}|\nabla v|^2\\
                                                &\leq&C_1.
\end{eqnarray*}
By means of Lemma \ref{lem42} and Lemma \ref{lem43} we thereby obtain that
\begin{equation}\label{stability}
u(\cdot,t)\to\bar{u}_0\quad\mbox{and}\quad v(\cdot,t)\to\bar{v}_0\quad \mbox{in}\ C^0(\bar{\Omega})
\end{equation}
as $t\to\infty$. Next we prove the decay property of $w$. Aided by (\ref{stability}), we have $T>0$ such that
\begin{equation*}
u(\cdot,t)\geq\frac{\bar{u}_0}{2}\quad\mbox{and}\quad v(\cdot,t)\geq\frac{\bar{v}_0}{2}
\end{equation*}
for all $t\geq T$. Then from the third equation in (\ref{CSC}), we get
\begin{equation*}
w_t\leq\Delta w-\frac{1}{2}(\alpha\bar{u}_0+\beta\bar{v}_0)w\quad\mbox{for all}\ t\geq T,
\end{equation*}
which yields
\begin{equation*}
w(\cdot,t)\leq\|w(\cdot,T)\|_{L^{\infty}(\Omega)}e^{-\frac{1}{2}(\alpha\bar{u}_0+\beta\bar{v}_0)(t-T)}\quad\mbox{for all}\ t\geq T,
\end{equation*}
and thereby completes the proof.$\hfill\Box$
\vskip 6mm
\noindent{{\bf{Acknowledgements.}}
This work is supported by National Natural Science Foundation of China (No. 11601127) and China Postdoctoral Science Foundation (No. 2018M630824).
%%%%%%%%%%%%%%%%%%%%%%%%%%%%%%%%%%%%%%%%%%%%%%%%%%%%%%%%%%%%%%%%%%%%%%%%%%%%%%%%

%\bibliographystyle{siam}
%\bibliography{Ref}

\end{document}